\documentclass[a4paper,10pt]{article}
\usepackage{amsmath,amssymb,amsfonts,amsthm}
\usepackage{graphicx}
\usepackage{pstricks}
\usepackage[colorlinks,citecolor=blue]{hyperref}
\usepackage[utf8]{inputenc}
\usepackage{enumitem}
\usepackage{textcomp}
\usepackage{caption}
\usepackage{xspace}

\newtheorem{theorem}{Theorem}[section]
\newtheorem{lemma}[theorem]{Lemma}

\newtheorem{corollary}[theorem]{Corollary}
\newtheorem{proposition}[theorem]{Proposition}
\newtheorem{probleme}[theorem]{Problem}
\newtheorem{convention}[theorem]{Convention}

\newtheorem{remarque}[theorem]{Remark}

\newcommand{\GGBS}{$vGBS$\xspace}
\newcommand{\vGBS}{$vGBS$\xspace}
\newcommand{\GBS}{$GBS$\xspace}
\newcommand{\GBSn}[1]{$GBS_{#1}$\xspace}

\title{Multiple conjugacy problem in graphs of  free abelian groups}

\author{Benjamin Beeker}

\renewcommand{\AA}{a}
\newcommand{\BB}{b}

\newcommand{\Z}{\mathbb Z }
\newcommand{\Q}{\mathbb Q}
\newcommand{\N}{\mathbb N }

\begin{document}
\maketitle

\begin{abstract}
A group $G$ is a \vGBS group if it admits a decomposition as a finite graph of groups with all edge and vertex groups finitely generated and free abelian. We prove that the multiple conjugacy problem is solvable between two $n$-tuples $A$ and $B$ of elements of $G$ whenever the elements of $A$ does not generate an elliptic subgroup.
\end{abstract}

\section{Introduction}

Let $\Gamma$ be a finite graph of groups with all vertex (and edge) groups free abelian of finite rank. Let $G$ be the fundamental group of $\Gamma$. We call such a group a Generalized Baumslag-Solitar group of variable rank, or \GGBS group. The graph $\Gamma$ is a \GGBS decomposition of $G$. When a \GGBS decomposition of a \GGBS group $G$ has all vertex and edge groups of a fixed rank $n$, we say that $G$ is a \GBSn{n} group or \GBS group if $n=1$.

The paper focuses on the following problem:

\begin{probleme}[Multiple conjugacy problem]
Let $G$ be a \GGBS group. Let $A=(\AA_1,\dots,\AA_n)$ and $B=(\BB_1,\dots,\BB_n)$ be two $n$-tuples of elements of $G$. Is there an element of $G$ which conjugates $A$ to $B$?
\end{probleme}


Given $\Gamma$ a graph of groups, with fundamental group $G$, an element $g\in G$ is \textit{elliptic} if it is conjugate into a vertex group, otherwise $g$ is \textit{hyperbolic}. In \cite{BoMaVen}, O. Bogopolski, A. Martino and E. Ventura show that the conjugacy problem is not decidable in \GBSn{n} groups with $n\geq 4$. The non-decidability is resulting from the fact that it is impossible to decide whether two elements of $\Z^4$ belong to the  same orbit of $\Z^4$ under the action of a finitely generated subgroup $\langle \varphi_1,\dots,\varphi_p\rangle$ of $Gl_4(\mathbb Z)$. Now take the \GBSn{4} group $F=\Z^4\rtimes_{(\varphi_1,\dots \varphi_p)}\mathbb F_p$,  two elements $a$ and $b$ of $\Z^4$ are conjugate  in $F$ if and only if there exists an element $\varphi$ in $\langle \varphi_1,\dots \varphi_p\rangle\subset Gl_4(\mathbb Z)$ such that $\varphi(a)=b$. Thus the conjugacy problem is unsolvable for elliptic elements in \GBSn{4} groups. 

However, this kind of counter-example is the only obstruction to the conjugacy problem, it is even the only obstruction to the multiple conjugacy problem in \vGBS groups. 
Actually, we have the following result.

\begin{theorem}\label{thintro}
Let $G$ be a \GGBS group with \GGBS decomposition $\Gamma$. Let $A=(a_1,\dots,a_n)$ and $B=(b_1,\dots,b_n)$ be two $n$-tuples of elements of $G$. Assume that the group $\langle a_1,\dots,a_n \rangle$ is not elliptic in $\Gamma$. Then we can decide whether $A$ and $B$ are conjugate in $G$ or not. If $A$ and $B$ are conjugate we explicit a conjugating element.
\end{theorem}

Let us give a sketch of the proof. The conjugacy problem between hyperbolic elements is solvable in groups much more general than \vGBS groups. We give in section \ref{sectionconjhyp} some conditions under which this problem is decidable.  

Now consider two $n$-tuples $A$ and $B$ of elements of $G$. If the subgroup $\langle A\rangle$ generated by the elements of $A$ is not elliptic then we may find a hyperbolic element $a\in \langle A\rangle$. We may construct $b\in\langle B\rangle$ such that if $A$ and $B$ are conjugate by an element $g$, then $gag^{-1}=b$.

If $b$ is not hyperbolic, then $A$ and $B$ cannot be conjugate. Otherwise we may decide whether $a$ and $b$ are conjugate or not. If $a$ and $b$ are not conjugate then again $A$ and $B$ are not conjugate. If there exists $g$ such that $gag^{-1}=b$, by conjugating $B$ by $g^{-1}$, we reduce theorem \ref{thintro} to solving the following problem. 

\begin{probleme}
Given two $n$-tuples $A$ and $B$, and a hyperbolic element $a$, does there exist an element $g$ of $C_G(a)$, the centralizer of $a$ in $G$, which conjugates $A$ to $B$?
\end{probleme}

For simplicity, we describe the case $n=1$, with $A=\{a_1\}$ and $B=\{b_1\}$. The case $n>1$ follows from the same techniques.

 We denote by $\mathcal T$ the Bass-Serre tree of $\Gamma$. If an element $h$ is elliptic, its \textit{characteristic space} is the set of its fixed points in $\mathcal T$. If an element $h$  is hyperbolic, its \textit{characteristic space}  is the unique line of $\mathcal T$ on which it acts by a translation, also called the \textit{axis} of $h$. The characteristic space of an element $h$ is denoted $\mathcal A_h$. The centralizer $C_G(h)$ of a hyperbolic element $h$ fixes  the axis of $h$ setwise, acting on it by a translation. To be more precise, we can decompose $C_G(h)$ as a semi-direct product $\Z^p\rtimes \Z$ such that the $\Z^p$ part fixes the axis pointwise while the $\Z$ part acts on it by a translation.

If there exists an element $g$ such that $ga_1g^{-1}=b_1$, then $g\cdot\mathcal A_{a_1}=\mathcal A_{b_1}$. This gives strong informations on $g$. In particular, if we assume in addition that $g$ belongs to the centralizer of a hyperbolic element $a$, we may obtain the translation length of $g$ on $\mathcal A_a$ by looking at the relative position of $\mathcal A_{a_1}$ and $\mathcal A_{b_1}$ with respect to $\mathcal A_a$ (see figure \ref{positionrelative} for an example).

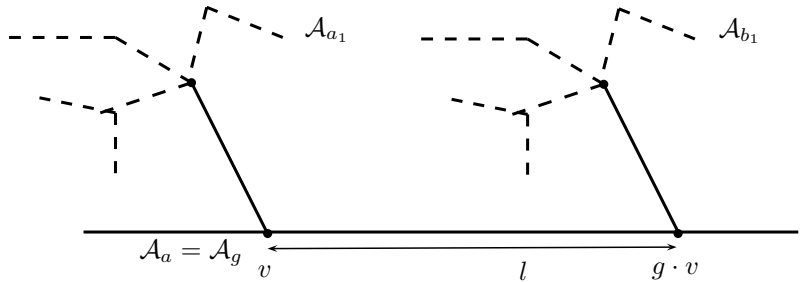
\begin{figure}

\begin{center}
\scalebox{1} 
{
\begin{pspicture}(0,-1.8789062)(11.854688,1.8589063)
\psline[linewidth=0.04cm](0.98,-1.1410937)(10.38,-1.1410937)
\usefont{T1}{ptm}{m}{n}
\rput(2.3942182,-1.4160937){$\mathcal A_a=\mathcal A_g$}
\psdots[dotsize=0.12](3.4,-1.1610937)
\psdots[dotsize=0.12](8.8,-1.1610937)
\psdots[dotsize=0.12](2.4,0.8389062)
\psline[linewidth=0.04cm](2.4,0.8389062)(3.4,-1.1610937)
\psline[linewidth=0.04cm,linestyle=dashed,dash=0.16cm 0.16cm](2.4,0.8389062)(1.2,0.43890625)
\psline[linewidth=0.04cm,linestyle=dashed,dash=0.16cm 0.16cm](2.4,1.0389062)(2.6,1.8389063)
\psline[linewidth=0.04cm,linestyle=dashed,dash=0.16cm 0.16cm](1.4,0.43890625)(1.4,-0.36109376)
\psline[linewidth=0.04cm,linestyle=dashed,dash=0.16cm 0.16cm](2.4,0.8389062)(1.4,1.4389062)
\psline[linewidth=0.04cm,linestyle=dashed,dash=0.16cm 0.16cm](1.4,0.43890625)(0.4,0.63890624)
\psline[linewidth=0.04cm,linestyle=dashed,dash=0.16cm 0.16cm](2.6,1.8389063)(3.6,1.4389062)
\psline[linewidth=0.04cm,linestyle=dashed,dash=0.16cm 0.16cm](1.4,1.4389062)(0.0,1.4389062)
\usefont{T1}{ptm}{m}{n}
\rput(4.184219,1.5439062){$\mathcal A_{a_1}$}
\psline[linewidth=0.02cm,arrowsize=0.05291667cm 2.0,arrowlength=1.4,arrowinset=0.4]{<->}(3.4,-1.3610938)(8.78,-1.3410938)
\usefont{T1}{ptm}{m}{n}
\rput(6.7642183,-1.6560937){$l$}
\usefont{T1}{ptm}{m}{n}
\rput(3.3628125,-1.6560937){$v$}
\usefont{T1}{ptm}{m}{n}
\rput(8.762813,-1.6560937){$g\cdot v$}
\psdots[dotsize=0.12](7.82,0.81890625)
\psline[linewidth=0.04cm](7.82,0.81890625)(8.82,-1.1810937)
\psline[linewidth=0.04cm,linestyle=dashed,dash=0.16cm 0.16cm](7.82,0.81890625)(6.62,0.41890624)
\psline[linewidth=0.04cm,linestyle=dashed,dash=0.16cm 0.16cm](7.82,1.0189062)(8.02,1.8189063)
\psline[linewidth=0.04cm,linestyle=dashed,dash=0.16cm 0.16cm](6.82,0.41890624)(6.82,-0.38109374)
\psline[linewidth=0.04cm,linestyle=dashed,dash=0.16cm 0.16cm](7.82,0.81890625)(6.82,1.4189062)
\psline[linewidth=0.04cm,linestyle=dashed,dash=0.16cm 0.16cm](6.82,0.41890624)(5.82,0.61890626)
\psline[linewidth=0.04cm,linestyle=dashed,dash=0.16cm 0.16cm](8.02,1.8189063)(9.02,1.4189062)
\psline[linewidth=0.04cm,linestyle=dashed,dash=0.16cm 0.16cm](6.82,1.4189062)(5.42,1.4189062)
\usefont{T1}{ptm}{m}{n}
\rput(9.614219,1.5239062){$\mathcal A_{b_1}$}
\end{pspicture} 
}

\end{center}
\caption{An element $g$ which acts on $\mathcal A_a$ such that $g\cdot \mathcal A_{a_1}=\mathcal A_{b_1}$  has translation length equal to $l$.}
\label{positionrelative}
\end{figure} 

We may determine this position:

\begin{theorem}
 
 Let $G$ be a \GGBS group with \GGBS decomposition $\Gamma$, let $a$ and $a_1$ be two elements of $G$, with  $a$ hyperbolic in $\Gamma$.
Call $\mathcal A_{a}$ and $\mathcal A_{a_1}$ their characteristic spaces. 

There exists an algorithm which decides 
\begin{itemize}
 \item whether ${\mathcal A}_a \cap {\mathcal A}_{a_1}= \emptyset$, and if so gives the shortest path between ${\mathcal A}_a$ and ${\mathcal A}_{a_1}$,
 \item or ${\mathcal A}_a \cap {\mathcal A}_{a_1}$ is non empty and of finite length, and if so gives its endpoints,
 \item or ${\mathcal A}_a \cap {\mathcal A}_{a_1}$ is a half-line, and if so gives the origin and the direction of the half-line,
 \item or ${\mathcal A}_a \cap {\mathcal A}_{a_1}={\mathcal A}_a$.
\end{itemize}
\end{theorem}
We give a proof of this theorem in section \ref{charac}.

As $C_G(a) \simeq \Z^p \rtimes \Z$ with $\Z^p$ fixing $\mathcal A_a$ pointwise, the $\Z$ coordinate of $g$ is fully determined as soon as ${\mathcal A}_a \cap {\mathcal A}_{a_1}\neq{\mathcal A}_a$ (if ${\mathcal A}_a \cap {\mathcal A}_{a_1}={\mathcal A}_a$,  the problem reduces to a problem of conjugacy in polycyclic groups which is solved \cite{BaCaRo}). We are then reduced to finding the $\Z^p$ coordinate, that is finding a conjugating element in a given elliptic subgroup.
We discuss this problem in part \ref{algo}.

%
%


In the case of \GBS groups, an idea of V. Guirardel and G. Levitt permits to solve the conjugacy problem for elliptic elements. We thus have the following result.
\begin{theorem}\label{thgbs} Let $G$ be a \GBS group and $\Gamma$ a \GBS decomposition of $G$.
The multiple conjugacy problem is decidable in $G$.
\end{theorem}

We give a proof of this theorem in section \ref{sectiongbs}. The conjugacy problem in \GBSn{2} and \GBSn{3} groups remains unknown.

\section{Preliminaries}\label{prelim}

\subsection{Conventions}

\paragraph{Presentation adapted to a graph of groups}\label{presad}
Let $G$ be a  finitely presented group and $\Gamma$ a finite graph of groups with $\pi_1(\Gamma)=G$.   Call $V_\Gamma$  and $E_\Gamma$ the sets of vertices and oriented edges of $\Gamma$. Given an edge $e\in E_\Gamma$, the opposite edge is denoted $\bar e$. We note $G_e$ the edge group of $e$. For a vertex $v\in V_\Gamma$ the vertex group is denoted $G_v$. For each edge $e$, call $\varphi^i_e$ and $\varphi^t_e$ the injections of $G_e$ into the groups of respectively the initial and terminal vertices of $e$. We have $G_e=G_{\bar e}$ and $\varphi^i_{\bar e}=\varphi^t_e$. We denote $\mathcal T$ the Bass-Serre tree of $\Gamma$.


Recall that an element of $G$ is \textit{elliptic} if it fixes a vertex of $\mathcal T$, and \textit{hyperbolic} otherwise. Given a hyperbolic element $g\in G$, it acts by translation on a line $\mathcal A_g$ of $\mathcal T$ called the \textit{axis} of $g$.  We define the \textit{positive direction} on $\mathcal A_g$ (depending on $g$), as the direction of the translation of $g$ on $\mathcal A_g$. A half-line contained in $\mathcal A_g$ is said to be  \textit{$g$-positive} if it is infinite in the positive direction and \textit{$g$-negative} otherwise.
A \textit{fundamental domain} of $g$ is a segment $[v,v']$ of $\mathcal A_g$ in $\mathcal T$  such that $g\cdot v=v'$.

The \textit{translation length} of an element $g$ is denoted $l(g)$.

Assume that $\Gamma$ has  finitely presented vertex and edge groups. 
 We fix a preferred generating set and a presentation of $G$ in the following way:
\begin{enumerate}
 \item we fix a maximal subtree $A$ of $\Gamma$, and call $E_A$ the set of edges of $A$,
 \item for each vertex $v$ of $V_\Gamma$, with vertex group $G_v$, we fix a presentation $\langle a_{v,1},\dots,a_{v,r_v} | r_{v,1}, \dots r_{v,s_v}\rangle$  of $G_v$,
 \item for each edge $e$  with initial vertex $v$ and edge group $G_e$, we fix $\tilde a_{e,1},\dots \tilde a_{e,r_e}$ a generating set of $G_e$ such that $\tilde a_{e,i}=\tilde a_{\bar e,i}$ We define  $a_{e,i}=\varphi_e^i(\tilde a_{e,i})$ and we express the $a_{e,i}$ as words in the generators of $G_v$.
\end{enumerate}

The preferred generating set of $G$ is $\left\{ a_{v,j},~v\in V_\Gamma,~j \in \N\right\}\cup \left\{t_e,~e\in E_\Gamma\right\}$. The element $t_e$ is called the stable letter associated to $e$. The stable letter $t_{\bar e}$ will also be written $\bar t_{e}$.

To obtain a presentation from this generating set, we add the following relations: 
\begin{enumerate}[itemsep=0cm,topsep=0cm]
 \item $\left\{ r_{v,1}, \dots r_{v,s_v}\right\},~\forall v\in V_\Gamma$, 
  \item $t_{\bar e}=t_e^{-1},~\forall e\in E_\Gamma$,
 \item $a_{e,j}=t_ea_{\bar e,j}t_e^{-1},~\forall e\in E_\Gamma~\forall j\in \N$,
 \item $t_{\bar e}=1,~\forall e\in E_A$.
\end{enumerate}

By \cite{Serre}, this is a presentation of $G$ and it will be called a \textit{presentation adapted to the decomposition $\Gamma$}. This presentation defines an injection of the vertex groups of $\Gamma$ into $G$.

If $\Gamma$ is a \vGBS decomposition, we suppose in addition that for each vertex $v$ and edge $e$ , the sets $\left\{a_{v,1},\dots,a_{v,r_v}\right\}$ and $\left\{\tilde a_{e,1},\dots,\tilde a_{e,r_e}\right\}$ are bases of $G_v\simeq \Z^{r_v}$ and $G_e\simeq \Z^{r_e}$.


\paragraph{Loop form}

A \textit{path} between two vertices $v$ and $v'$ in $\Gamma$ is a sequence $(v=v_0,e_1,v_1,\dots,e_n,v_n=v')$ where the $v_i$ are vertices of $\Gamma$ and the $e_i$ are edges of $\Gamma$ with starting point $v_{i-1}$ and endpoint $v_i$. The integer $n$ is the \textit{length} of the path. A \textit{loop} is a path between a vertex $v$ and itself.

Similarly, a \textit{path} in $\mathcal T$  is a sequence $(v_0,e_1,v_1,\dots,e_n,v_n)$ where the $v_i$ are vertices of $\mathcal T$ and the $e_i$ are edges of $\mathcal T$ with initial vertex $v_{i-1}$ and terminal vertex $v_i$.
  
 Let $G$ be a \GGBS  group given by a presentation adapted to a \GGBS decomposition $\Gamma$. Let  us fix a \textit{base vertex} $v_0$ of $\Gamma$. 
 
 Let  $l=(v_0,e_1,\dots, e_{p},v_p=v_0)$ be a loop of $\Gamma$ whose base point is $v_0$ and let $g$ is an element of $G$. 
A \textit{loop form} of $g$ (of basis the loop $l$) is a word $c_0t_1c_1t_2\dots t_pc_p$ representing $g$ in $G$, with each $c_i$ a -possibly trivial- element of the vertex group $G_{v_i}$, and $t_i$ the stable letter associated to the edge $e_i$.The terms $c_i$ are called \textit{vertex terms} and the $t_i$ are called \textit{edge terms}. The integer $p$ is the \textit{length of the loop form}.


If $l$ is fixed, a loop form of $g$ of basis $l$ may not exist. However, there always exists some loop form of~$g$:

\begin{lemma} \label{colsimple}
Let $g$ be an element of $G$ given as a word in the preferred generators of $G$. We may explicitly construct a loop form of $g$.

Moreover elements which fix a common vertex in $\mathcal T$ admit loop forms with a common base loop.
\end{lemma}

\proof
Let $g_1=a_1t_1\dots a_p$ and $g_2=b_1s_1\dots b_q$ be two elements expressed in some loop form. 
Then a loop form of $g_1g_2$ is $a_1t_1\dots (a_pb_1)s_1\dots b_q$, and a loop form of $g_1^{-1}$ is $a_p^{-1}\bar t_p\dots \bar t_1a_1^{-1}$.

Thus, to construct a loop form for $g$ it is sufficient to construct one for the stable letters  and
 the generators of the vertex groups.
 
 Let $t$ be the stable letter associated to an edge $e$. Call $v$ and $v'$the initial and the terminal vertices of $e$. Let $(v_0,e_1,\dots, e_{p},v_p=v)$ and $(v_0,e'_1,\dots, e'_{q},v'_q=v')$ be paths between $v_0$ and $v$, and $v_0$ and $v'$ in $A$. Then a loop form for $t$ is $1 t_{e_1} 1 \dots 1 t_{e_p}1t1\bar t_{e'_q}\dots \bar t_{e'_1} 1$.
 
Let $v$ be a vertex of $\Gamma$ and  $(v_0,e_1,\dots, e_{p},v_p=v)$ be a path between $v_0$ and $v$ in $A$. Then any element $a$ of $G_v$ admits a loop form $1 t_{e_1}1\dots 1 t_{e_p}at_{\bar e_p}\dots t_{\bar e_1}1$. 

We may notice that in both cases every element $t_{e_i}$ or $t_{\bar e_i}$ is trivial.

The moreover part follows immediately from the previous paragraph, if we notice \begin{itemize}[noitemsep]\item that if $a$ and $b$ have loop forms with a common base loop, then the same holds for $gag^{-1}$ and $gbg^{-1}$ and \item that elements which fix a common vertex in $\mathcal T$ may be simultaneously conjugated into the group of a vertex of $\Gamma$.\hfill$\qedsymbol$\end{itemize}

A loop form $c_0t_1c_1t_2\dots t_pc_p$ is a \textit{reduced form} if moreover we have the implication 
$e_i=\bar e_{i+1} \Rightarrow c_i\not\in \varphi_t^{e_i}(G_{e_i}).$

\begin{lemma}[{\cite[Theorem 11]{Serre}}]\label{formenormale}
 The trivial element has a unique reduced form, the basis of this loop form is the trivial loop $(v_0)$ of length $0$.
\end{lemma}


In general a reduced form is not unique. We can however exhibit a kind of uniqueness:

\begin{lemma}\label{baseloop}
 If $g$ admits two reduced forms, their base loops are equal.
\end{lemma}

\begin{proof}
 Take $c_0t_1c_1t_2\dots t_pc_p$ and $c'_0t'_1c'_1t'_2\dots t'_qc'_q$ the two reduced form. As the product of the first expression by the inverse of the second $$c_0t_1c_1t_2\dots t_p(c_pc'^{-1}_q)t'^{-1}_qc'^{-1}_{q-1}t'^{-1}_{q-1}\dots t'^{-1}_1c'^{-1}_0$$ is a loop form (up to replacing $t'^{-1}_i$ by $\bar t'_i$) representing the trivial element, by lemma \ref{formenormale}, this
loop form is not reduced. But the only simplification may occur at $t_pc_pc'^{-1}_qt'^{-1}_q$. 

So $t_p=t'_q$ and $c_pc'^{-1}_q$ belongs to the associated edge group. We obtain a new loop form. By recurrence on the length $p+q$, we must have $p=q$, and the loops must be identical.
\end{proof}

Moreover in \GGBS decompositions we may explicitly construct a reduced form from a loop form. 
As the matter of fact, we may algorithmically construct reduced forms whenever $G$ is the fundamental group of a graph of groups $\Gamma$ (with finitely presented vertex and edge groups) whenever the following membership problem (P1) is solvable:
\begin{probleme}[P1]
Let $e$ be an edge of $\Gamma$ with initial vertex $v$. Given $g\in G_v$, does $g$ belong to $\varphi_e^i(G_e)$?
\end{probleme}

Given a loop form $c_0t_1c_1t_2\dots t_pc_p$, we may compute a reduced form using the following process.  If there exists an $i$ such that $e_i=\bar e_{i+1}$  and $c_i\in \varphi_t^{e_i}(G_{e_i})$, we may find $d_i\in G_{e_i}$ such that $c_i=\varphi_t^{e_i}(d_i)$, we then replace $c_{i-1}t_ic_it_{i+1}c_{i+1}$ by $(c_{i-1}d_ic_{i+1})$ in the loop form.
Since the edge groups of $\Gamma$ are finitely presented, it permits us here to express the preimage by $\varphi_e^i$ of any element in $\varphi_e^i(G_e)$ in term of generators of $G_e$.

Thanks to lemma \ref{formenormale}, we can assert that the word problem is solvable whenever (P1) and the word problem for vertex groups are solvable.
 

\subsection{Algorithmic generalities}

We fix $\tilde A$ a lift of $A$ in $\mathcal T$ such that the lift $\tilde v$ of a vertex $v$ in $\tilde A$ has stabilizer $G_v$.
For each edge $e$ in $\Gamma\setminus A$, with initial and terminal vertices $v$ and $v'$, we fix a lift $\tilde e$ in $\mathcal T$ defined as the edge between $\tilde v$ and $t_e\cdot \tilde v'$.
We may then identify each vertex $v$ and edge $e$ of $\Gamma$ with a vertex $\tilde v$ and an edge $\tilde e$ of $\mathcal T$.

Note that the lift of $\bar e$ is not $\bar {\tilde e}$ but $t_{\bar e}\bar {\tilde e}$.

\begin{convention}\label{sommet}A vertex $\tilde w$ of $\mathcal T$ is now seen as a product $g\cdot v$ with $g$ an element of $G$ and $v$ a vertex of $\Gamma$, such that $\tilde w=g\cdot \tilde v$.  We may notice that $g$ is not unique. 
\end{convention}

The following algorithmic property gives us a kind of constructibility of the Bass-Serre tree $\mathcal T$:

\begin{proposition}\label{constructionarbre1}
 Let $G$ be a group with decomposition $\Gamma$. Assume that (P1) is solvable in $\Gamma$. Let $g_v$ and $g_v'$ be two elements given as words in preferred generators of $G$. Let $v=g_v\cdot v_i$ and $v'=g_{v'}\cdot v_j$ be two vertices of $\mathcal T$. We may compute the distance between $v$ and $v'$ in $\mathcal T$ and construct the minimal path between them in the tree. 
 
 Given any finite set $\mathcal V$ of vertices of $\mathcal T$, we may explicitly construct the convex hull of $\mathcal V$ in $\mathcal T$.
 \end{proposition}

\begin{proof}
By lemma \ref{colsimple}, we may construct loop forms $c_0t_1c_1t_2\dots t_pc_p$ of $g_v$ and $d_0s_1d_1s_2\dots s_qd_q$ of $g_{v'}$. Let $l=(v_0,e_1,\dots, e_{p},v_p=v_0)$ and  $l'=(w_0=v_0,f_1,\dots, f_{q},w_q=v_0)$ be the base loops of these two loop forms. 

Let $(v_0,\gamma_1,y_1,\gamma_2,\dots,y_k=v_i)$  and $(v_0,\delta_1,z_1,\delta_2,\dots,z_l=v_i)$ be the paths between $v_0$ and $v_i$, and $v_0$ and $v_j$ in $\tilde A$. Then a path between $v$ and $v'$ is the concatenation of the four following paths 
\begin{enumerate}[itemsep=0cm]
\item $(g_v\cdot v_i=g_v\cdot y_k, g_v\cdot \bar \gamma_k,\dots, g_v\cdot \bar \gamma_1, g_v\cdot v_0)$, 
\item $(g_v\cdot w_q=(c_0t_1c_1t_2\dots t_p)\cdot v_0,(c_0t_1c_1t_2\dots c_{p-1}t_p)\cdot \bar e_p,\dots, c_0t_1\cdot \bar e_{1}, c_0 \cdot w_0=v_0)$
\item $(v_0,d_0\cdot f_1, d_0s_1\cdot w_1, d_0s_1d_1\cdot f_2,  d_0s_1d_1s_2\cdot w_2,\dots, d_0s_1d_1s_2\dots s_q\cdot w_q=g_{v'}\cdot v_0)$ 
\item $(g_{v'}\cdot v_0,g_{v'}\cdot\delta_1,g_{v'}\cdot z_1, \dots,g_{v'}\cdot v_j=v')$.
\end{enumerate}

Let us now reduce this path. A path is reduced  if for every triplet $(\tilde e=g_e\cdot e,\tilde w= g_w\cdot w, \tilde f=g_f\cdot f)$ such that $\bar f = e$ appears in the path, we have $\tilde e\neq \bar {\tilde f}$, that is $g_f^{-1}g_e\not \in G_w$. The problem $(P1)$ allows us to check whether the element $g_f^{-1}g_e\in G_w$ belongs to $G_f$ or not. If $g_f^{-1}g_e\in G_f$ we delete the triplet and the immediately following vertex (which is the same as the immediately preceding one). 
Iterating the process, we obtain the minimal path between $v$ and $v'$. 

We obtain the convex hull of a set of vertices $\mathcal V$, by constructing the path between each pair of vertices of $\mathcal V$, and then glue the paths together (by checking if some vertices of these paths are equal).
\end{proof}

\begin{corollary}\label{constructionarbre} Assume that $(P1)$ is decidable.
 \begin{enumerate}
 \item Given $g\in G$, we may compute $l(g)$. If $g$ is elliptic ($l(g)=0$) we may construct a fixed point, and if $g$ hyperbolic, we may  construct a fundamental domain. 
 \item Given a vertex $v$, we may decide if $v$ belongs to the characteristic space of $g$.
\end{enumerate}
\end{corollary}

\begin{proof}
Let $g$ be an element of $G$ and $v$ a vertex of $\mathcal T$.
By proposition \ref{constructionarbre1} we may construct the minimal path $\mathcal P$ between $v$ and $g\cdot v$.
Now $\mathcal P$ contains a fixed point if $g$ is elliptic, or a fundamental domain if $g$ is hyperbolic. Constructing for each vertex $w$ in $\mathcal P$ the path between $w$ and $g\cdot w$ permits to determine whether $g$ is elliptic or hyperbolic: if one of these paths is of length $0$, the element $g$ is elliptic, otherwise $g$ is hyperbolic and the constructed path of minimal length is a fundamental domain.

For the second part, to see if a vertex $v$ belongs to the characteristic space, it suffices to check if the path between $v$ and $g\cdot v$ has length $l(g)$.
\end{proof}


\section{\texorpdfstring{Algorithmic properties of \GGBS groups}{}}\label{algo}

Until now, we have given general properties of graphs of groups. We now describe algorithmic properties specific to \vGBS groups.

By lemma \ref{formenormale} and the construction of a reduced form, the word problem is solvable in \GGBS groups. However, the vertex groups are free abelian and so the algorithmic properties are much stronger. It is thus possible to solve certain equations in \vGBS groups:

\begin{lemma}\label{lemmetechnique}
Let $G$ be a \GGBS group. Let $g_1,\dots, g_n, a_0, \dots, a_n$ be elements of $G$ with all $a_i$ elliptic. Let $p$ be an integer and $\sigma : [|0,n|] \rightarrow [|1,p|]$ a map. Then the set 
$$K=\{\underbar{k}= (k_1,\dots, k_p) | a_0^{k_{\sigma(0)}}g_1a_1^{k_{\sigma(1)}}g_2\dots g_na_n^{k_{\sigma(n)}}=1\}$$
is a finite union of affine sublattices of $\mathbb Z^p$, \textit{i.e.} a finite union of sets of the form $\left\{x|x=a+g,g\in R\right\}$ where $a$ is an element of $\Z^p$ and $R$ a subgroup of $\Z^p$.

The set $K$ may be algorithmically described.
\end{lemma}

\proof All elements which may be written under the form $a_0^{k_{\sigma(0)}}g_1\dots g_na_n^{k_{\sigma(n)}}$, have loop forms with a common base loop: for a given $i$ the elements of the form $a_i^{k_{\sigma(i)}}$ all belong to the same vertex group, and thus by lemma \ref{colsimple} have loop forms with a base loop which does not depend on the $\underbar{k}$. Thus it suffices to construct a loop form to each $g_i$  and then to glue the different part together (see lemma \ref{colsimple}).

We obtain a common loop form $$a_0^{k_{\sigma(0)}}g_1a_1^{k_{\sigma(1)}}g_2\dots g_na_n^{k_{\sigma(n)}}=w_0(\underline{k})t_1w_1(\underline{k})\dots w_{i-1}(\underline{k})t_iw_i(\underline{k})\dots t_mw_m(\underline{k})$$
where the $t_i$ are edge terms and the $w_i$ vertex terms depending on $\underbar k$. Moreover if $v_i$ is the vertex associated to $w_i$, there exist $p+1$ elements $b_{i,0},\dots b_{i,p}$ (non necessarily distinct) which belong to the vertex group $G_{v_i}$ and such that we have the equality $w_i(\underline{k})=b_{i,0}\prod_{j=1}^p b_{i,j}^{k_{j}}$.

The lemma is now reduced to the following equivalent fact:

Given a loop $l=(v_0,e_1,\dots, e_m,v_m)$ of length $m$, and elements $(b_{i,j})$ for $i\in[|0,m|]$ and  $j\in [|0,p|]$ with $b_{i,j}\in G_{v_i}$, we may algorithmically describe the set 
$$K=\{\underbar{k}= (k_1,\dots, k_p) |b_{0,0}(\prod_{j=1}^pb_{0,j}^{k_j})t_1b_{1,0}(\prod_{j=1}^pb_{1,j}^{k_j})t_2\dots t_mb_{m,0}(\prod_{j=1}^pb_{m,j}^{k_j})=1\}$$ as a finite union of sublattices of $\Z^p$.

The proof will be performed by recurrence on $m$. For the need of the recurrence, we only suppose that there exists $r\in \N$ such that the $b_i$ belongs to $\frac{1}{r}G_{v_i}=G_{v_i}\otimes_\Z \frac{1}{r}\Z$, the set of vectors of which coefficients are fractions with $r$ as denominator. We may then see the $w_i$'s as maps from $\Z^p$ to $\frac{1}{r}G_{v_i}$. For each $w_i$ call $\tilde w_i$ the morphism defined by $\tilde w_i(\underbar k)=\prod_{j=1}^p b_{i,j}^{k_{j}}$. Unlike $w_i$, the morphism $\tilde w_i$ is a group homomorphism.

The goal is now to construct the set of $p$-tuples $\underline k=(k_1,\dots k_p)\in \Z^p$ such that for all $i$, the element $w_i(\underline{k})$ belongs to $G_{v_i}\subset \frac{1}{r}G_{v_i}$ and such that the equality
 $$w_0(\underline{k})t_1w_1(\underline{k})\dots w_{i-1}(\underline{k})t_iw_i(\underline{k})\dots t_mw_m(\underline{k})=1$$ is verified. We also have to show that this set is an affine sublattice of $\Z^p$. The lemma is the case $r=1$.

By recurrence on the length $m$ of the base loop of $w$.
\begin{enumerate}[leftmargin=0cm]
 \item If $m=0$. We have to find the set $K$ of $\underline k=(k_1\dots k_p)\in\Z^p$ such that $w_0(\underline{k})=b_{0,0}\prod_{j=1}^p b_{0,j}^{k_{j}}=1$.

The set of solutions is just the set $w_0^{-1}(\{1\})$. Computing it consists in solving the linear equation $b_0\prod_{j=1}^p b_j^{k_{j}}=1$ in $\Z^p$. We deduce that the set is of the form $\left\{x|x=a+g,g\in R\right\}$ where $a$ is a solution and $R$ is the kernel of $\tilde w_0$.

Thus, in this case, the solutions forms an affine sublattice which may be explicitly described.

 \item If $m>0$.

Using lemma \ref{formenormale} , if  $w(k_1,\dots,k_p)$ is trivial, there exists $i$ such that $t_i$ is associated to an edge $e_i$ and $t_{i+1}$ to $\bar e_i$ and such that $w_i \in \varphi_{e_i}^i(G_{e_{i}})$ the associated edge group.

Let $i$ be such that $t_i$ is associated to an edge $e_i$ and $t_{i+1}$ to $\bar e_i$. We first compute $ S_i=\{\underline{k} | w_i(\underline{k})\in G_{e_{i}}\}$.

Let us call $\pi_i$ the projection of $G_{v_i}$ onto $G_{v_i}/\varphi_{e_i}^i(G_{e_i})$.  Then $S_i$ may be put under the form $\left\{x|x=a+g,g\in R\right\}$ where $a$ is a solution of the equation $\pi_i\circ w_i =1$ and where $R$ is the kernel of $\pi_i\circ\tilde w_i$. Thus $S_i$ is an affine sublattice of $\Z^p$ that we may explicitly describe.

Let $V_j$ be a vector space such that $G_{v_{j}}\otimes\Q=(\varphi^t_j(G_{e_{j}})\otimes\Q)\oplus V_j$ and let us define a morphism $\psi_j:G_{v_{j}}\otimes \Q\rightarrow G_{v_{j-1}}\otimes Q$ in the following way:

For $g\in\varphi^t_{e_j}(G_{e_{j}})$, the element $\psi_j(g)=(\varphi^i_{e_j})_{\otimes \Q}\circ(\varphi^t_{e_j})_{\otimes \Q}^{-1}(g)$, and if $g\in V_j$ then $\psi_j(g)=1$.
Note that if $\varphi^i_{e_j}$ and $\varphi^t_{e_j}$ are given as matrices, then $\psi_j$ may be explicitly determined as a matrix.

Let $J \subset [|1,p-1|]$, the set of $j$ such that $e_j= -e_{j+1}$. Then the set of solutions $K$ is
$$\bigcup_{j\in J}
S_j\cap \{\underline{k} |w_0(\underline k)t_1\dots t_{j-1}(w_{j-1}(\underline{k})\psi_j(w_j(\underline{k}))w_{j+1}(\underline{k}))t_{j+2}\dots t_mw_m(\underline k)=1\}.$$

Note that if $\underline k$ belongs to $S_j$, then the element $\psi_j(w_i(\underline{k})))$ 
 belongs to $G_{v_{j-1}}$ (and not only to $G_{v_{j-1}}\otimes \Q$) .

To be able to apply the recurrence hypothesis, we must check that we may write the expression $w_{j-1}(\underline{k})\psi_j(w_i(\underline{k}))w_{j+1}(\underline{k})$ under the form $c_0\prod_{i=1}^p c_i^{k_{i}}$ with the $c_i$ in $\frac{1}{r'}G_{v_{j-1}}$ for an $r'\in \N$.
We know that there exists $r$ such that all values of the $w_j$ belongs to $\frac{1}{r}G_{v_j}$.
The map $\psi_j$  sends $\frac{1}{r}G_{v_j}$ to $\frac{1}{rr''}G_{v_{j-1}}$  where $r''$ is the index $[\varphi^i_{e_j}(G_{e_j}),G_{v_{j-1}}]$.

We may then write $w_{i-1}(\underline{k})\varphi_i(w_i(\underline{k}))w_{i+1}(\underline{k})$ as $c_0\prod_{j=1}^p c_j^{k_{j}}$, with $c_i$ in $\frac{1}{rr'}G_{v_{i-1}}$.
We conclude by recurrence.

Since the intersection of two sublattices of $\mathbb Z^p$ is a constructible sublattice, the set of solutions may be given in an explicit manner.\hfill$\qedsymbol$
\end{enumerate} 

\begin{corollary}\label{conjugaisonlocale}
Let $g$ and $g'$ be two elements of $G$ and $v$ a vertex of $\mathcal T$. Then the set of elements of $G_v$ which conjugate $g$ in $g'$ is an affine sublattices of $G_v$ which may be describe explicitly. In particular, for $g=g'$, the intersection of the centralizer of $g$ and $G_v$ is computable.
\end{corollary}

\begin{proof}
 Let $a_0, \dots, a_p$ be the preferred basis of $G_v$. Then the set of elements conjugating $g$ in $g'$ are the elements of the form $\Pi a_i^{k_i}$ such that $\Pi a_i^{k_i}g\Pi a_i^{-k_i}g'^{-1}=1$. Since the centralizer of $g$ in $G_v$ is a subgroup of $G_v$,  this set is an affine sublattice of $G_v$. We may construct this set applying lemma \ref{lemmetechnique}.
\end{proof}

\section{Conjugacy problem for hyperbolic elements}\label{sectionconjhyp}

 Let $\Gamma$ be a graph of groups of which vertex and edge groups are finitely presented.
 
 If two elements $a$ and $b$ are conjugate by a third element $g$, then $g$ sends the characteristic space of $a$ to the one of $b$. If $a$ and $b$ are hyperbolic their characteristic spaces are lines and are very simple to describe. Knowing these spaces gives us information on $g$.

Fix a preferred presentation for $G$ the fundamental group of $\Gamma$ as described in section \ref{prelim}. Call $\mathcal T$ the Bass-Serre tree of $\Gamma$. We assume that every element of $G$ is given as a word in the preferred generators, and that every vertex is given following the convention \ref{sommet}.

\begin{lemma}\label{debutconj}
Assume that the following problems are solvable:
\begin{itemize}
\item[(P1')] Let $\tilde e$ be an edge of $\mathcal T$ with an endpoint $\tilde v$. Given $g\in G_{\tilde v}$, does $g$ belong to $G_{\tilde e}$ ?
\item[(P2)]  Let $\tilde v$ and $\tilde v'$ be two vertices of $\mathcal T$ and $g\in G$ an elliptic element. Is $G_{\tilde v}\cap (G_{\tilde v'}\cdot g)$ empty ?
\end{itemize}
Given two hyperbolic elements $a$ and $b$ and $n\in \N$, we may decide if there exists $g\in G$ such that the intersection of the axes $\mathcal A_a$ and $\mathcal A_{gbg^{-1}}$ has length at least $n$. 
\end{lemma}

The problem (P1') is equivalent to the problem (P1) expressed in the Bass-Serre tree rather than in the graph of groups : given a vertex $\tilde v=h\cdot v$ of initial vertex $\tilde e=h\cdot e$ and an element $g$, we have $g\in G_{\tilde e}$ if and only if $h^{-1}gh\in G_e$.

\begin{proof}
Applying corollary \ref{constructionarbre}, we may construct fundamental domains $\mathcal D_a$ and $\mathcal D_b$ associated to $a$ and $b$.

A pair of segments $([v,v'],[w,w'])$ is a \textit{good pair} if
\begin{itemize}[topsep=0cm, itemsep=0cm]
\item $v$ belongs to $D_a$, 
\item $[v,v']$ is included in the axis of $a$,
\item $w$ belongs to $D_b$, 
\item $[w,w']$ is included in the axis of $b$,
\item there exists a $g\in G$ sucht that $g\cdot [v, v']=[w,w']$
\end{itemize}

We may notice that there exists $g\in G$ such that $\mathcal A_a\cap \mathcal A_{gbg^{-1}}$ contains a segment $[v,v']$ of length $n$, if and only if there exists two integers $p$ and $q$ such that $(a^p\cdot [v,v'], b^qg\cdot [v,v'])$ is a good pair of segments of length $n$.
We thus have to decide if such a good pair exists.
 
 Note that if $([v,v'],[w,w'])$ is a good pair of segments, then truncating both segments gives a new good pair of segments $([v,v''],[w,w''])$. We then say that $([v,v'],[w,w'])$ \textit{extends} $([v,v''],[w,w''])$. If two good pairs of segments have same length and extend a common good pair of segments, then they are equal. Said differently, if a good pair of segments admits an extension of a fixed length, this extension is unique.

By recurrence, we show that we may construct all good pairs of segments of length $n$. And in particular, we may know if there exists one.

If $n=0$, then a good pair of segments is just a couple of vertices $(v,w)\in \mathcal D_a\times \mathcal D_b$ such that $v$ and $w$ are in the same orbit. Since vertices are given by the product of an element of $G$ and the representative in $\Gamma$ of its orbit, all of these good pairs are obtained directly.

Suppose now that we have constructed all good pairs of segments of length $n$.

For every good pair of segments $([v,v'],[w,w'])$ of length $n$, we first find $g$ such that $g\cdot [v,v']=[w,w']$. We may proceed by an exhaustive search.

Call $v''$ and $w''$ the vertices of $\mathcal A_a$ and $\mathcal A_b$ respectively such that $[v,v'']$ is a segment of length $n+1$ in $\mathcal A_a$ and $[w, w'']$ is a segment of length $n+1$ in $\mathcal A_b$. Call $e$ the edge between $v'$ and $v''$, and $e'$ the edge between $w'$ and $w''$.

The only candidate which may extend $([v,v'],[w,w'])$ is $([v,v''],[w,w''])$. We thus have to decide whether this last pair is a good pair or not.

If $e$ and $e'$ are not in the same orbit, then there is no good pair of segments of length $n+1$ extending $([v,v'],[w,w'])$.

If $e$ and $e'$ are in the same orbit, find $h$ such that $hg\cdot e =e'$. Then $h$ belongs to $G_{w'}$, and is consequently elliptic.
An element $g'$ sends $[v,v'']$ on $[w, w'']$ if and only if $g'g^{-1}$ belongs to $G_w$ and $g'g^{-1}h^{-1}$ belongs to $G_{w''}$. It is sufficient to decide if the intersection $G_w\cap G_{w''}\cdot h$ is empty. This problem is solvable by hypothesis.

We may thus construct all good pairs of segments of length $n+1$.
\end{proof}

Let $\Gamma$ be a graph of group with Bass-Serre tree $\mathcal T$. The tree $\mathcal T$ is said to be $k$-acylindrical if the stabilizer of any segment $[v,v']$ in $\mathcal T$ of length $k$ is trivial, \textit{i.e.} the intersection $G_v\cap G_{v'}$ is trivial. It is said acylindrical if it is $k$-acylindrical for some $k$. Acylindricity is defined by Sela in \cite{Sela97}.

\begin{corollary}
Let $\Gamma$ be a graph of groups of Bass-Serre tree $\mathcal T$ in which (P1') and (P2) are solvable.
If $\mathcal T$ is $k$-acylindrical for a given $k\in\N$ then the conjugacy problem is decidable for hyperbolic elements of $G$.
\end{corollary}

\begin{proof}
Let $a$ and $b$ be two hyperbolic elements of $G$. We are computable $l(a)$ and $l(b)$ thanks to corollary \ref{constructionarbre}. 
If $l(a)\neq l(b)$ then $a$ and $b$ are not conjugate, else fix $n=k+l(a)$.

Applying lemma \ref{debutconj}, we may decide if there exists $g$ such that $\mathcal A_a$ and $\mathcal A_{gbg^{-1}}$ have an intersection of length $n$. Obviously if no such $g$ exists then $a$ and $b$ are not conjugate.
If such a $g$ exists, then $a^{-1}gbg^{-1}$ fixes a segment of which length is at least $k$ which is thus trivial, and then $a$ and $b$ are conjugate.
\end{proof}

In a \GGBS decomposition, the problems (P1') and (P2) are easily solvable:\begin{itemize}[noitemsep, topsep=0cm]
\item the problem (P1') consists in understanding if an element of $\Z^p$ belongs to a subgroup;
\item the problem (P2) consist in successively computing the intersection of subgroups of $\Z^p$.
\end{itemize}
  However \GGBS decompositions are never acylindrical. To solve the conjugacy problem for hyperbolic elements we need to know how to solve another problem:

\begin{proposition}\label{conj}
 Let $\Gamma$ be a graph of groups for which (P1') and the local conjugacy problem (P3) are solvable: 
 \begin{itemize}
\item[(P3)] Given a vertex $v$  and two hyperbolic elements $a$, $b$, is there $g\in G_v$ such that $gag^{-1}=b$?
\end{itemize}
Then the conjugacy problem for hyperbolic elements is solvable in $G=\pi_1(\Gamma)$.
\end{proposition}

By corollary \ref{conjugaisonlocale}, the problem (P3) is solvable for \GGBS groups.

\begin{corollary}The conjugacy problem for hyperbolic elements is solvable in \GGBS groups.\hfill$\qedsymbol$
\end{corollary}
\begin{proof}[Proof of proposition \ref{conj}]
Let $a$ and $b$ be two hyperbolic elements of $G$. Call $\mathcal T$ the Bass-Serre tree of $\Gamma$.

As (P1') is solvable,  applying corollary \ref{constructionarbre}, we may construct a vertex $v$ which belongs to the axis of $a$ and $\mathcal D_b$ a fundamental domain of $b$. 
Assume that $a$ and $b$ are conjugate by an element $g\in G$. Then for $w$ in the axis of $a$, the vertex $g\cdot w$ belongs to the axis of $b$. Up to multiplying $g$ on the left by a well-chosen power of $b$, we may assume that $g\cdot v$ belongs to $\mathcal D_b$.

Let $D_v = \left\{ v_{i_1},\dots,v_{i_p}\right\}$ be the vertices of $\mathcal D_b$ in the $G$-orbit of $v$. Choose $g_1,\dots, g_p$ some elements of $G$ such that $g_j\cdot v = v_{i_j}$.

Then $a$ and $b$ are conjugate if and only if there exists  $j\in [|1,p|]$ and $h$ in the stabilizer of $v_{i_j}$ such that $hg_jag_j^{-1}h^{-1}=b$. Thus it suffices to apply for each element of $\mathcal D_v$ the local conjugacy problem that we have assumed to be solvable.
%
\end{proof}

\paragraph{Reduction of the problem}~

We now look at the multiple conjugacy problem in \GGBS  groups.

\begin{probleme} \label{multi1}
 Let $G$ be a \GGBS. 
Given two $(n+1)$-tuples $A=(a_0,\dots,a_n)$ and $B=(b_0,\dots,b_n)$, does there exist an element of $G$ which conjugates $A$ in $B$ ?
\end{probleme}

For the same reasons as for the conjugacy problem, the general problem is not solvable. We have to restrict it to the case where the group $G_A$ generated by the elements of  $A$ contains a hyperbolic element. 
\begin{remarque}\label{changementth}
By a lemma of Serre \cite[Proposition 26]{Serre}, either $G_A$ is elliptic or there exist two elements $a_i$ and $a_j$ such that $a_ia_j$ is hyperbolic. Then adding this element to $A$, and at the same time adding the element $b_ib_j$ to $B$, we may assume that $A$ contains a hyperbolic element.
\end{remarque}

Using proposition \ref{conj}, the theorem \ref{thintro} is equivalent to solving the following problem:

\begin{probleme} \label{multi2}
Given two $n$-tuples $A$ and $B$ and a hyperbolic element $a$, does there exist an element of $C_G(a)$ the centralizer of $a$ in $G$ which conjugates $A$ to $B$? 
\end{probleme}

We solve this problem in two steps. In section \ref{centralizer}
 we give an explicit description of 
$C_G(a)$. In section \ref{charac} we determine the position of the characteristic spaces of the $a_i$ and $b_i$ relative to the one of $a$, in the way to decide if they are compatible the one with the others.

\section{Centralizer of hyperbolic elements}\label{centralizer}

If two elements $a$ and $b$ commute, then $b$ preserves the characteristic space of $a$. Unfortunately, it is hard to describe the characteristic space of an elliptic element. It is easy to describe the centralizer of a hyperbolic element using the fact that its characteristic space is a line:

\begin{proposition}\label{centralisateur}
Let  $\Gamma$ be a graph of groups and $G$ its fundamental group. Let $h$ be a hyperbolic element of $G$. Then $C_G(h)$, the centralizer of $h$ in $G$, is a semi-direct product $E\rtimes H$ where $E$ is an elliptic subgroup and $H$ a cyclic subgroup generated by a hyperbolic element.

If $\Gamma$ is a \GGBS decomposition, the centralizer of a hyperbolic element is of the form  $\mathbb{Z}^n\rtimes\mathbb{Z}$.  Moreover in this case, we may explicitly give a generating set.
\end{proposition}

\proof
By definition, the group  $C_G(h)$ commutes with $h$, and thus acts on its axis by translation  in the Bass-Serre tree. This action defines a map from $C_G(h)$ to $\mathbb Z$ with image $p\Z$ for $p$ the minimal non-zero translation length in $C_G(h)$. The kernel $E$ of this map fixes the axis pointwise, and then is included in an edge group. We can take for $H$ any section of $p\Z$.

If $G$ is a  \GGBS group, we proceed in two steps  to explicitly determine $C_G(h)$.
\begin{itemize}[leftmargin=0cm]
 \item Step 1: Computation of $E$.

Let $v$ be a vertex of the axis of $h$. The group $E$ fixes the axis, in particular it is included in $G_v$. This group is the set of elements in $G_v$ which commute with $h$. Corollary  \ref{conjugaisonlocale} explicitly gives the sublattice of $G_v$ consisting of elements which commute with $h$.

\item Step 2: Compute a generator for $H$.

From the first part of the lemma, it suffices to find a hyperbolic element $h'$ of $C_G(h)$ with minimal translation length. Let $\mathcal D = (v_0,\dots,v_d)$ be a fundamental domain of $h$. Then the desired element $h'$ satisfy $h'\cdot v_0=v_k$ for some $k\leq d$.

We determine $K\subset [|1,d|]$ the set of elements such that $v_k$ and $v_0$ are in the same $G$-orbit. Then for all $k\in K$ we compute $g_k$ such that $g_k\cdot v_0=v_k$.

Now, we search $h'$ of the form $h'=g_kg$ for a $k\in K$ and a $g$ in $G_{v_0}$. By lemma \ref{lemmetechnique}, for every $k\in K$ we may compute the set $G_k$ of elements $g$ in $G_{v_0}$ such that $[g_kg,h]=1$. Note that $G_d$ is not empty since $hg^{-1}_d$ belongs to $G_d$. Let $k_0$ be the least integer such that $G_{k_0}\neq \emptyset$.  We may take for $h'$ any element in $g_{k_0}G_{k_0}$.~\hfill$\qedsymbol$
\end{itemize}

\section{\texorpdfstring{Modulus of a hyperbolic element in a \GGBS group}{}}\label{modulus}

Let $M$ be a finitely generated free abelian group  seen as a $\Z$-module. Let $V$ and $W$ be two subgroups of $M$ and $\varphi : V \rightarrow W$ an isomorphism of groups.

Denote $\varphi_{\otimes\mathbb Q} : V\otimes\mathbb Q \rightarrow W\otimes\mathbb Q$ the extension of $\varphi$, where $V\otimes\mathbb Q$ and  $W\otimes\mathbb Q$ are seen as subspaces of $M\otimes \mathbb Q$. Call $D_{\varphi}$ the set of elements $x \in (V\cap W)\otimes \mathbb Q$ verifying $$\exists (x_n)_{n\geq 0} \in((V\cap W)\otimes \mathbb Q)^{\mathbb N} \text{ such that } \varphi_{\otimes\mathbb Q}(x_{n+1})=x_n \text{ and } x=x_0$$ the $\Q$-subspace of $M\otimes \Q$ composed of elements $x$ of $V\otimes \mathbb Q$ such that $\varphi^{-n}_{\otimes \mathbb Q}(x) $ is defined for all $n\in \N$. We define the $\Q$-linear map $\tilde \varphi = (\varphi_{\otimes \mathbb Q})_{ | D_{\varphi}}$.


The minimal polynomial of a linear map $\varphi$ will be for us the monic polynomial $P$ of least degree such that $P(\varphi)=0$.

\begin{proposition} \label{alg} With the same notations as before, 
 \begin{enumerate}[leftmargin=0cm]

  \item \label{point1} the map $\tilde \varphi$ is an automorphism of $D_{\varphi}$.  
  We may see $D_{\varphi}$ as the set of elements $x$ of $M\otimes\Q$ such that $\varphi_{\otimes\Q}^{n}(x) $ is defined for all $n\in \Z$ (or for all $n\in \N$).
\item \label{point2}For all $k\neq 0$ we have $D_{\varphi^k}=D_{\varphi}$ and $\widetilde{(\varphi^k)} = (\tilde \varphi)^k$.
  \item Let $\varphi$ and $\varphi'$ be two homomorphisms respectively of sources $V\subset M$ and $V'\subset M$. If $V$ and $V'$ are commensurable and $\varphi_{|V\cap V'}=\varphi'_{|V\cap V'}$, then $D_\varphi= D_{\varphi'}$ and $\tilde \varphi = \tilde \varphi'$,

  \item \label{alg3} for $x\in M$, we have $x \in D_\varphi$ if and only if $\forall k \in \mathbb Z~\exists p \in \mathbb N$ such that $\varphi^k(x^p)$ is defined.

  \item \label{droite} Let $x$ be in $M$. There exists $p \in \mathbb N$ such that  $\varphi^k(x^p)$ is defined for all $k\in \mathbb N$ if and only if  there exists a subspace $D_x \subset D_\varphi$ containing $x$ and $\tilde \varphi$-invariant, such that the  minimal polynomial of $\tilde \varphi_{|D_x}$ belongs to $\Z[X]$.

  \item Let $x$ be in $M$. There exists $p \in \mathbb N$ such that $\varphi^{-k}(x^p)$ is defined for all $k\in \mathbb N$ if and only if  there exists a subspace $D_x \subset D_\varphi$ containing $x$ and $\tilde \varphi$-invariant, such that the  minimal polynomial of $\tilde \varphi^{-1}_{|D_x}$ belongs to $\Z[X]$.
 \end{enumerate}
\end{proposition}
\proof 
\begin{enumerate}[itemindent=1cm, leftmargin=0cm]
\item By construction, if $x\in D_\varphi$ then there exists a sequence $(x_n)_{n\in \N}$ such that $\varphi_{\otimes Q}(x_{n+1})= x_n$ and $x_0=x$. In particular $x_1 \in D_\varphi$.  Thus $\varphi^{-1}_{\otimes \Q}$ is an isomorphism from $W\otimes \Q$ to $V\otimes Q$ which send $D_\varphi$ into $D_\varphi $ with $D_\varphi$ of finite dimension. Hence the linear map $\tilde \varphi$ is an isomorphism.
\end{enumerate}
\begin{enumerate}[resume,leftmargin=0cm,topsep=0cm]
\item From point \ref{point1}, we have $D_\varphi=D_{\varphi^{-1}}$. We thus prove point \ref{point2} only for $k\in \N$. 
For all $k\in \N$ the sequence $(y_n)=(x_{kn})$ is such that $\varphi^k(y_{n+1})=y_n$ and $y_0=x$. Hence  $D_\varphi \subset D_{\varphi^k}$.

Conversely, if there exists a sequence  $(y_n)_{n\in\N}$ such that $\varphi^k(y_{n+1})=y_n$ and $y_0=x$, we may extend the sequence in a sequence $(x_n)$  define by $x_0=x$ and 
$x_{(kp-q)}=\varphi^q(y_p)$ with $p\in \N$ and $q\in[|0,k-1|]$. As for all $p\in \N$, the element $\varphi^k(y_p)$ is defined, the elements $\varphi^q(y_p)$ for $q\in[|0,k-1|]$ are also defined. The sequence $(x_n)$ verifies $\varphi(x_{n+1})=x_n$ and $x_0=x$. In particular $D_\varphi = D_{\varphi^k}$.

\item As $V$ and $V'$ are commensurable, $V\otimes\Q = V'\otimes\Q = (V\cap V') \otimes \Q$. In particular $\varphi_{\otimes \Q}=(\varphi_{|V\cap V'})_{\otimes\Q}=(\varphi'_{|V\cap V'})_{\otimes\Q}=\varphi'_{\otimes \Q}$. Hence $D_{\varphi}=D_{\varphi'}$ and $\tilde \varphi=\widetilde {(\varphi')}$.

\item We proved that the $\Q$-vector space $D_{\varphi^k}$ does not depend on $k$, however the definition set $V_k$ of $\varphi^k$  (which is a subgroup of $D_{\varphi^k}$) depends on $k$. Note that $V_1=V$, $V_{-1}=W$ and for all $k\geq 0$ we have $V_{k+1}\subset V_k$ and $V_{-k-1}\subset V_{-k}$.

As $D_{\varphi^k}/(V_k\cap D_{\varphi^k})$ is a torsion group, if $y \in D_{\varphi^k}$ then there exists $p$ such that $y^p \in V_k$, then $\varphi^k(y^p)$ is defined. If $x\in D_\varphi$, then for all $k\in \mathbb Z$, the element $x$ is in $D_{\varphi^k}=D_\varphi$. Thus, for all $k$ there exists $p$ such that $\varphi^k(x^p)$ is defined.

Conversely, assume that for all $k$ there exists an integer $p_k$ such that $x^{p_k}$ belongs to $V_k$. Call $x_k=\varphi^{k}(x^{p_{k}})$. and define the sequence $\tilde x_k=x_{k}^{\frac{1}{p_{k}}}$ of element in $V\otimes \Q$. We have $\varphi_{\otimes \mathbb Q}(x_{k})=x_{k+1}$ and $x_0=x$. Thus by definition $x$ belongs to $D_\varphi$.

\item Call  $\widehat D_x = \langle \varphi^k(x^p), k \in \mathbb N \rangle \subset  M$ the $\Z$-module generated by the iterated images of $x^p$. By construction $\varphi$ is an endomorphism of $\widehat D_x$. In particular the minimal polynomial of $\varphi_{|\widehat D_x}$ has its coefficients in $\Z$. Moreover as $\varphi(\widehat D_x)\subset \widehat D_x$, we have $\widehat D_x\subset D_{\varphi}$. Hence $D_x:=\widehat D_x\otimes \Q$ is a $\tilde \varphi$-invariant subspace of $D_\varphi$. The minimal polynomial of $\tilde \varphi_{|\widehat D_x}$ which is equal to the minimal polynomial of $\varphi_{|\widehat D_x}$ has its coefficients in $\Z$. 

Conversely, suppose there exists a $\tilde \varphi$-invariant subspace $D_x$ of  $D_\varphi$ containing $x$ and such that the minimal polynomial of $\tilde \varphi_{|D_x}$ has its coefficients in $\Z$. The matrix of the restriction of $\tilde \varphi$ to $D_x$ has integer entries in the basis associated to its invariant factors. On the one hand, there exists $p$ such that all $p$th powers of the elements of this basis belongs to $V$, on the other hand there exists $m$ such that $x^m$ has integer coordinates in this basis. Then $p=mq$ is as required.
\item It suffices to apply the previous point to $\varphi^{-1}$.\hfill$\qedsymbol$
\end{enumerate}
\paragraph{Modulus of hyperbolic elements}Let $G$ be a \GGBS  group, let $h$ be a hyperbolic element. Fix $w_0$ a vertex of the axis of $h$. Define $M=G_{w_0}$, $V = G_{w_0} \cap h^{-1} G_{w_0}h= G_{w_0} \cap G_{h^{-1}\cdot w_0}$ and $W = G_{w_0} \cap h G_{w_0}h^{-1}=G_{w_0} \cap G_{h\cdot w_0}$. Then $t_h$ the conjugation by $h$ defines an isomorphism between $V$ and $W$. 
We define $\varphi_h = \tilde t_h$ the \textit{modulus} of $h$ (relative to $w_0$). It is an automorphism of $D_{t_h} \subset G_{w_0}\otimes \Q$. We take the notation $D_h$ instead of $D_{t_h}$.

\begin{remarque}\label{remmod}
By point $4$ of proposition \ref{alg}, an element $g\in G_{w_0}$ belongs to $D_h$ if and only if for every $n\in \Z$ there exists an integer $p$ such that $t^n_h(g^p)\in G_{w_0}$, that is, 
for every $n\in\Z$ there exists an integer $p$ such that $g^p$ belongs to $G_{h^{-n}\cdot w_0}$.
\end{remarque}

\begin{lemma}\label{propmodule}
\begin{enumerate}  [itemindent=-0.4cm]
\item The modulus $\varphi$ is an element of $Gl(D_h)$ and for all $k\in \mathbb Z^*$ the modulus of $h^k$ is $(D_h,\varphi^k_h)$.
\item Let  $(D'_h,\varphi'_h)$ the modulus of $h$ constructed from an other vertex $w'_0$ of the axis. Then $\mathcal D = D_h\cap G_{w_0}$ and $\mathcal D' = D'_h\cap G_{w'_0}$ are commensurable and $\varphi_{h|\mathcal D\cap \mathcal D'}=\varphi'_{h|\mathcal D\cap \mathcal D'}$. In particular, we may identify canonically $(D_h,\varphi_h)$ and  $(D'_h,\varphi'_h)$.
\item The set $D_h$ and  the map $\varphi_h$ are explicitly constructible.
\end{enumerate}
\end{lemma}

\begin{proof}
The point $1$  is the transcription in $G$ of point $1$ of the proposition  \ref{alg}.

By the remark \ref{remmod}, the intersections $D_h\cap G_{w_0}$ and $D'_h\cap G_{w'_0}$ are commensurable. Thus the point $2$ is a consequence of point $2$ of proposition \ref{alg}.

We now construct $D_h$ and  $\varphi_h$. 
We may assume that the group is given by a preferred presentation (cf section \ref{presad}).
First, with corollary \ref{constructionarbre}, we may construct a fundamental domain $[w_0,h\cdot w_0]$ of $h$ and its projection to $\Gamma$.  Note that the preferred presentation gives us a basis of the image of the edge groups in the vertex groups. As $V$ and $W$ may be seen as intersections of edge groups, we may construct two bases $\mathcal B$ and $\mathcal B'$ of $V$ and $W$. We may also compute a matrix representing $t_h:V\rightarrow W$ in the bases $\mathcal B$ and $\mathcal B'$. Moreover, we may iteratively construct the sequence $V_i$ of subgroups of $G_{w_0}$ defined by $V_0=V$ and $V_i=t_h(V_{i-1})\cap V_{i-1}$. The sequence $V_i\otimes \Q$ is first strictly decreasing (for the inclusion) and then stationary. The set $D_h$ is thus equal to the first term $V_i\otimes \Q$ such that $V_i\otimes \Q=V_{i+1}\otimes \Q$. Then $\varphi_h=t_{h|D_h}\otimes \Q$.
\end{proof}

\section{Characteristic spaces}\label{charac}

 Given $w_0$ a vertex of the axis of $h$, we define $\widehat V_h = G_{w_0} \cap D_h$, which does not depend on the choice of $w_0$ up to commensurability. Let $g$ be an elliptic element of $G$. If there exists $n\in \N$ such that $g^n\in\widehat V_h$, then we note $g\in[D_h]$. 
By proposition \ref{alg} we have $g\in [D_h]$  if and only if $\forall k \in \mathbb Z,\exists p \in\mathbb N,~g^p\in G_{h^k\cdot w_0}$.

\begin{proposition} \label{geom}
Let $G$ be a \GGBS  group. Let $h$ be a hyperbolic element with axis $\mathcal A$ and let $g$ be an elliptic element. For every $p$, denote $\mathcal E_p$ the set of fixed points of $g^p$. 

There exists $p$ such that $\mathcal E_p\cap\mathcal A$ is a $h$-positive half-line if and only if $g \in [D_h]$ and there exists a $\varphi_{h}$-stable subspace $D_g$ of  $D_h$ such that $g$ has a power in $D_g\cap G_{w_0}$ and such that the minimal polynomial of $\varphi_{h|D_g}$ is in  $\mathbb Z[X]$.

\end{proposition}

\begin{proof}
By definition $\mathcal E_p\cap\mathcal A$ is a $h$-positive half-line if and only if there exists $k_0$ such that for all $k\geq k_0$,  the element $g^p$ belongs to the group $G_{h^k\cdot w_0}$, that is, $\forall k\in\N$ we have $h^{-k}g^ph^k \in G_{w_0}$. 
Thus, the proposition results directly from the point \ref{droite} of proposition \ref{alg}.
\end{proof}

\begin{corollary}\label{finitude}
With the notations of proposition \ref{geom}, call $s$ the translation length of $h$. If $g^p$ fixes a $h$-positive half-line and $g$ fixes a segment of the axis of length $d=s\cdot\dim D_g$ then $g$ fixes a $h$-positive half-line.
\end{corollary}

\begin{proof} Suppose $g^p$ fixes a $h$-positive half-line. Let $D_g$ be as in proposition \ref{geom}. We may assume that $D_g$ is chosen of minimal dimension. Define $r=dim D_g$ its dimension. If $g$ fixes a segment of the axis of length $d$, then for all $0 \leq k< r$, the elements $\varphi^k(g)$ all fix a same vertex $v$ of the axis. 
Moreover $(\varphi^k(g))$ generate $D_g(\subset G_v\otimes \Q)$, thus $(\varphi^k(g))_{k<r}$ is a basis of $D_g(\subset G_v\otimes \Q)$. The matrix of $\varphi_{|D_g}$ in this basis is  $\left(
\begin{array}{ll}
0 & a_0\\
Id_{r-1} & A                                                                                                                                                                                                                                                                               \end{array}
\right)$
with $a_0$ an element of $\Q$ and $A$ a column vector of length $r-1$ of elements of $\Q$. From proposition \ref{geom}, the minimal polynomial of $\varphi_{|D_g}$ is in $\Z[X]$, thus the coefficients of the vector  $A$ and the element $a_0$ belong to $\Z$. That is $\varphi_{|D_g\cap G_v}$ has its image in $D_g\cap G_v$. Hence, for all $k\in \N$ we have $\varphi^k(g)\in D_g\cap G_v$. Thus $g$ fixes a $h$-positive half-line.
\end{proof}

This proposition is algorithmic. More precisely:
\begin{proposition}\label{positions}
Given a hyperbolic element $h$ with axis $\mathcal A_h$ and  any element $g$ with characteristic space $\mathcal A_g$, we may decide is which case $g$ fall :
\begin{itemize}
\item The intersection $\mathcal A_g\cap\mathcal A_h$ is finite (possibly empty),
\item the intersection $\mathcal A_g\cap\mathcal A_h$ is a positive half-line of $\mathcal A_h$,
\item the intersection $\mathcal A_g\cap\mathcal A_h$ is a negative half-line of $\mathcal A_h$,
\item $\mathcal A_g=\mathcal A_h$.
\end{itemize}
Moreover, if $\mathcal A_g\cap\mathcal A_h$ is empty, we may compute the shortest path between the two spaces. If this intersection is finite, we may compute this intersection. If this intersection is a half-line (with a known direction), we may compute its origin.
\end{proposition}

\proof
The moreover part is a direct consequence of proposition \ref{constructionarbre1} and corollary \ref{constructionarbre}:
if the intersection $\mathcal A_g\cap\mathcal A_h$ is empty, we find one vertex of $A_g$ and one of $A_h$ and construct the path between them. We may then reduce this path until no edge belong to $\mathcal A_g$ or $ \mathcal A_h$. The obtained path is the shortest path between $A_g$ and $A_h$.
If the intersection is not empty, using the same method we may find a vertex $v$ which belongs to $\mathcal A_g\cap\mathcal A_h$. If $\mathcal A_g\cap\mathcal A_h$ is finite in a given direction of $A_h$, we may run trough $A_h$ starting from $v$ in this direction until its extremity.

Let us prove now the main part of the proposition.
Let $v$ be a vertex in $\mathcal A_g$, and $w$ a vertex in $\mathcal A_h$. Call $d$ the translation length of $h$ on $\mathcal A_h$. 
Using proposition \ref{constructionarbre1}, we first construct the path $[v,w]$ between $v$ and $w$. 
Applying corollary \ref{constructionarbre}, we may check if  $[v,w]\subset \mathcal A_g\cup \mathcal A_h$.
If $[v,w]\not\subset \mathcal A_g\cup \mathcal A_h$ then the two characteristic spaces are disjoint.

 If $[v,w]\subset \mathcal A_g\cup \mathcal A_h$, 
 we have to distinguish the elliptic and the hyperbolic cases.
\begin{enumerate}[leftmargin=0cm, topsep=0cm]
\item First assume $g$ is elliptic. 
Call $v'$ a vertex belonging to the intersection $\mathcal A_g\cap \mathcal A_h$. The element $g$ belongs to $G_{v'}$. 

Let $D_h$ be  the definition set of $\varphi_h$ in $G_{v'}\otimes \Q$. 

By corollary \ref{propmodule}, we may construct $D_h\subset G_{v'}\otimes \Q$. Call $n$ its dimension.
We may decide if $g\in [D_h]$,.

If $g\not\in [D_h]$ then  $\mathcal A_g\cup \mathcal A_h$ is finite. 

If $g\in D_h$, we compute the smallest $\varphi_h$-invariant subspace $D_g=\langle \varphi^i(g), i\in[|1,n|]\rangle$ of $D_h$ containing $g$, and the minimal polynomial $P_g$ of $\varphi_{h|D_g}$. Then, applying proposition \ref{geom} and corollary \ref{finitude}, the element $g$ fixes a positive half line if and only if $P_g$ belongs to $\Z[X]$ and $g$ fixes a segment of length $d\cdot dim D_g$. As the characteristic space is connected, checking that $g$ fixes a segment of length $d\cdot dim D_g$ containing $v'$ is sufficient.

To check if $g$ fixes the negative half-line, it suffices to repeat the process interchanging $h$ and $h^{-1}$.

The four cases are the following:
\begin{itemize}[topsep=0cm]
\item the element $g$ does not fix neither a positive nor a negative half-line of $\mathcal A_h$, then  $\mathcal A_g\cap\mathcal A_h$ is finite;
\item  the element $g$ fixes a positive half-line, but no negative half-line, then $\mathcal A_g\cap\mathcal A_h$ is a positive half-line;
\item the element $g$ fixes a negative half-line, but no positive half-line, then $\mathcal A_g\cap\mathcal A_h$ is a negative half-line;
\item the element $g$ it fixes both a positive and a negative half-line, then it fixes globally $\mathcal A_h$.
\end{itemize}

\item Assume now that $g$ is hyperbolic. Call $d'$ the translation length of $g$. Construct $v'$ a vertex of the intersection $\mathcal A_g\cap\mathcal A_h$. We then check if the intersection $\mathcal A_g\cap\mathcal A_h$ has length $>d+d'$.

If $\mathcal A_g\cap \mathcal A_h$ has length less than $d+d'$, then  it is trivially finite. If its length is greater than $d+d'$, call $g'=[g,h]$ the commutator of $g$ and $h$. Then $g'$ fixes a vertex of $\mathcal A_g\cap \mathcal A_h$ and thus is elliptic. Moreover $\mathcal A_g\cap \mathcal A_h$ and $\mathcal A_{g'}\cap \mathcal A_h$ have Hausdorff distance less than $d+d'$. 
As we may decide if $\mathcal A_{g'}\cap \mathcal A_h$ is finite, a positive half-line, a negative half-line or the whole axis $ \mathcal A_h$ ,  we may also decide it for $\mathcal A_g\cap \mathcal A_h$.\hfill$\qedsymbol$
\end{enumerate}

Take $h$, $g$ and $g'$ three elements of $G$ with $h$ hyperbolic, with characteristic spaces $\mathcal A_h$, $\mathcal A_g$ and $\mathcal A_{g'}$. Suppose that $\mathcal A_g\cap \mathcal A_h$ and $\mathcal A_{g'}\cap \mathcal A_h$ have the same  form in the sens of proposition \ref{positions}, we may define the \textit{shift length between $g$ and $g'$} in the following way:
\begin{itemize}[itemsep=0cm]
\item if $\mathcal A_g\cap \mathcal A_h$ and $\mathcal A_{g'}\cap \mathcal A_h$ are both empty. Let $v$ and $w$ be the closest vertices of $A_h$ to respectively $A_g$ and $A_{g'}$, then the shift length between $g$ and $g'$ is the distance between $v$ and $w$;
\item if $\mathcal A_g\cap \mathcal A_h$ and $\mathcal A_{g'}\cap \mathcal A_h$ are both finite, assume $\mathcal A_g\cap \mathcal A_h=[v,v']$ and $\mathcal A_{g'}\cap \mathcal A_h=[w,w']$ with $v$ and $w$ the endpoints of the segments in the negative direction of $\mathcal A_h$,  the shift length between $g$ and $g'$ is the distance between $v$ and $w$;
\item if $\mathcal A_g\cap \mathcal A_h$ and $\mathcal A_{g'}\cap \mathcal A_h$ are both half-lines of the same direction, the shift length between $g$ and $g'$ is the distance between the origin of the two half-lines;
\item if $\mathcal A_g\cap \mathcal A_h=\mathcal A_{g'}\cap \mathcal A_h=A_h$, the shift length is not defined.
\end{itemize}

By proposition \ref{constructionarbre1}, the shift length is computable.

\section{Multiple conjugacy problem}\label{multiple}

Now let us solve problem \ref{multi2}, which we proved to be equivalent to theorem \ref{thintro}.

\begin{theorem}\label{multiconj}
 Let $G$ be a \GGBS group, $A=(a_1,\dots,a_n)$ and $B=(b_1,\dots,b_n)$ two $n$-tuples of $G$, and $a$ an hyperbolic element. We may decide whether $A$ and $B$ are conjugate by an element of $C_G(a)$, the centralizer of $a$ in $G$, or not.
\end{theorem}

\begin{proof}

By proposition \ref{centralisateur}, we may explicitly give $C_G(a)$ under the form $E \rtimes <h>$. Thus, the new problem is now:
\begin{center}\textit{Is there $\alpha\in E$ and $m\in \Z$ such that $\alpha h^mAh^{-m}\alpha^{-1}=B$?}\end{center}
We proceed in two steps : first finding $m$ then determining  $\alpha$.
 
Denote $l$ the translation length of $h$ on its axis $\mathcal A_h={\mathcal A}_{a}$. Every element of $E$ fixes ${\mathcal A}_h$ pointwise.
If there exists $g\in C_G(a)$ which conjugates $a_i$ to $b_i$, then it sends the characteristic space of $a_i$ to the one of $b_i$. In particular the intersections ${\mathcal A}_{a}\cap {\mathcal A}_{a_i}$ and ${\mathcal A}_{a}\cap {\mathcal A}_{b_i}$ differ by a translation.

Applying proposition \ref{positions}, we may determine ${\mathcal A}_{a}\cap {\mathcal A}_{a_i}$ and ${\mathcal A}_{a}\cap {\mathcal A}_{b_i}$ for every $i$. If each intersection is equal to the whole axis ${\mathcal A}_h$, then the group $G_{A,B}$ generated by the $a_i$, the $b_i$ and $C_G(a)$ is polycyclic (because it acts on a line with free abelian stabilizers). As $A$ and $B$ are conjugate in $G$ if and only if they are conjugate in $G_{A,B}$ then the problem is reduced to multiple conjugacy in polycyclic groups which is solvable \cite{GrunSe,BaCaRo}. 

Otherwise, choose $i$ such that the intersection ${\mathcal A}_{a}\cap {\mathcal A}_{a_i}$ differs from ${\mathcal A}_{a}$. Call $\mathcal F$ the form of this intersection, which is either "finite", "positive half-line" or "negative half-line".
Compute the intersection ${\mathcal A}_{a}\cap {\mathcal A}_{b_i}$. If the form of this intersection is not $\mathcal F$, then $A$ and $B$ are not conjugate.

If ${\mathcal A}_{a}\cap {\mathcal A}_{b_i}$ has form $\mathcal F$, then compute $d$ the shift length between $a_i$ and $b_i$. If $l$ does not divide $d$ then no element of $C_G$ send  ${\mathcal A}_{a}\cap {\mathcal A}_{b_i}$ to ${\mathcal A}_{a}\cap {\mathcal A}_{b_i}$ thus $A$ and $B$ are not conjugate.



If $l$ divides $d$, after conjugating $B$ by $h^{\frac {d}{l}}$, we have the equivalence: 
\begin{center}\textit{$A$ and $B$ are conjugate in $C_G(a)$ if and only if they are conjugate by an element of $E$.}\end{center}
By lemma \ref{lemmetechnique}, we may compute the sets of elements $E_i\subset E$ which conjugate $a_i$ to $b_i$. Every $E_i$ is a union of sublattice of $E$. Then $A$ and $B$ are conjugate if and only if the intersection $\bigcap_{i=2}^{n} E_i$ (which is computable) is not empty. Thus the problem is algorithmically solvable.
\end{proof}

\section{\texorpdfstring{Case of \GBS groups}{}}\label{sectiongbs}

The case of  \GGBS groups whose vertex and edge groups are cyclic is different. These groups are called \GBS groups.
In the case of \GBS groups, the conjugacy problem is solvable even for elliptic elements.

For hyperbolic elements, the problem is solved by theorem \ref{conj}. The last case to study is the one of elliptic elements. The method consists in reducing the problem into the reachability problem in Petri Nets which is solvable according to \cite{Reut}.

First, let us solve the problem for \GBS groups with a \GBS decomposition with one vertex. Let $G$ be one of these groups. Call $v$ the vertex and $a\in G$ an element generating $G_v$. Let $k$ be the number of loops in the graph of groups. The group admits a presentation $G=\langle a, t_1,\dots, t_k | t_ia^{\sigma_i}t_i^{-1}=a^{\tau_i}\rangle$. To introduce symmetry in the presentation, we introduce $k$ new generators $t_{k+1},\dots, t_{2k}$, and $2k$ relation $t_{i+k}=t_i^{-1}$, and $ t_ia^{\sigma_i}t_i^{-1}=a^{\tau_i},i\in[|k+1,2k|]$ where $\sigma_{i+k}=\tau_i$ et $\tau_{i+k}=\sigma_i$. Note that this gives a preferred presentation (see section \ref{presad} for the definition).

An elliptic element of $G$ is obviously conjugate to an element of the subgroup $\langle a \rangle$. Moreover, we may find explicitly an element of $\langle a \rangle$ in its conjugacy class (\textit{e.g.} by an exhaustive enumeration). 
We may thus restrict the conjugacy problem for elliptic elements to the conjugacy problem for elements in $\langle a \rangle$.

\begin{proposition}\label{Casun}
The conjugacy problem in $G$ is solvable.
\end{proposition}

\begin{proof}
Let $g=a^m$ and $g'=a^n$ be  two elements of $\langle a \rangle$. 
We prove that $g$ and $g'$ are conjugate in $G$ if and only if there exists a finite sequence of integers $(m_0=m,m_1,\dots, m_s=n)$ such that for all $i<s$ there exists a integer $j_i$ verifying $\sigma_{j_i} | m_i$ and $m_{i+1}= m_i \frac{\tau_{j_i}}{\sigma_{j_i}}$. Call such a sequence a \textit{conjugating sequence}.


If there exists a conjugating sequence $(m_0=m,m_1,\dots, m_s=n)$. Directly, we have the equality $t_{j_s}t_{j_{s-1}}\dots t_{j_1}g t_{j_1}^{-1}t_{j_2}^{-1}\dots t_{j_s}^{-1}=g'$. Thus $g$ and $g'$ are conjugate.

If $g$ and $g'$ are conjugate. Let $c$ be such that $cg'c^{-1}=g$. Then $g$ fixes both vertices $v$ and $c\cdot v$ in the Bass-Serre tree. Thus, it fixes the path $[v,c\cdot v]$. Let  $(v_0=v,v_1,\dots, v_s= c\cdot v)$. Then there are some generators $a_i$ of $G_{v_i}$ and integers $j_i$ such that $a_0 = a$ and $a_i^{\sigma_{j_i}}=a_{i+1}^{\tau_{j_i}}$.
Note that $a_s=cac^{-1}$.

Take the sequence $(m_0,\dots, m_s)$ of integers such that $g=a_i^{m_i}$. Then as $a_i^{m_i}=a_{i+1}^{m_{i+1}}$ we have $\sigma_{j_i} | m_i$ and $m_{i+1}= m_i \frac{\tau_{j_i}}{\sigma_{j_i}}$. Since $g=a^m$  we have $m_0=m$ and since  $cg'c^{-1}=ca^nc^{-1}$ , and $a_s=cac^{-1}$, we have $m_s=n$. 
The sequence $(m_0,\dots, m_s)$ is as required.

We reduced the conjugacy problem to the problem of finding such a sequence a sequence $(m_0,\dots, m_s)$. Remark that if such a sequence exists, the prime divisors of the $m_i$'s belongs to the (finite) set of prime divisors $\mathcal P=\{p_1,\dots, p_k\}$ of the elements $\tau_j$, $\sigma_j$, $n$ and $m$. Call $\pi$ the natural bijection between integers whose prime divisors are in $\mathcal P$ and $\N^k\times \Z/2\Z$ (which associate to an integer its decomposition in prime factors and its sign).

Finding a conjugating sequence $(m_0=n,\dots, m_s=m)$ is now equivalent to finding a sequence $(r_0=\pi(n),\dots, r_s=\pi(m))$ of elements of $\N^k\times \Z/2\Z$ such that for all $i<s$ there exists an integer $j_i$ verifying $r_i-\pi(\sigma_{j_i}) \in \N^k\times \Z/2\Z$ and $r_{i+1}= r_i +\pi(\tau_{j_i})- \pi({\sigma_{j_i}})$.

The problem to decide if such a sequence exists is proven to be solvable in \cite{Reut} (the $\Z/2\Z$ does not appear in the proof of \cite{Reut} but may easily be added). 
\end{proof}
Using the same method we may prove the conjugacy problem is solvable in any \GBS group:

\begin{theorem}\label{conjgbs}
 The conjugacy problem in \GBS groups is solvable.
 \end{theorem}

\begin{proof}Let $G=\pi_1(\Gamma)$ be a  \GBS group with set of vertices $\mathcal V=\{v_1, \dots v_r\}$, we denote $a_i$ a generator of the group $G_{v_i}$. Let $g$ and $g'$ be two elliptic elements which belongs to the groups of two vertices $v$ and $v'$. Let $m$ and $n$ be such that $g=b^m$ and $g'=b'^n$ where $b$ generates $G_v$ and $b'$ generates $G_{v'}$.

Let  $e_1, \dots, e_r$ be the edges of  $\Gamma$. Denote $\sigma_j$ and  $\tau_j$ the integers  such that the relation associated to $e_j$ is $t_i a^{\sigma_j}_s t^{-1}_i = a_t^{\tau_j}$. Call $i(e_i)$ and $t(e_i)$ the initial and terminal vertex of $e_i$.

We prove that $g$ and $g'$ are conjugate in $G$ if and only if there exists a finite sequence $((m_0,w_0)=(m,v),(m_1,w_1),\dots, (m_s,w_s)=(n,v'))$ of elements of $\Z\times \mathcal V$ such that for all $i<s$ there exists a edge $e_{j_i}$ verifying  $i(e_{j_i})=w_i$, $t(e_{j_i})=w_{i+1}$ and $\sigma_{j_i} | m_i$ and $m_{i+1}= m_i \frac{\tau_{j_i}}{\sigma_{j_i}}$. By extension we also call such a sequence a \textit{conjugating sequence}.

The proof is the same as in the previous proposition.

Let $\mathcal P=\{p_1,\dots, p_k\}$ be the set of prime divisors of $\sigma_j$, $\tau_j$, $m$ and $n$. Let $\mathcal Q=\{q_1,\dots, q_r\}$ be  a set of $r$ prime numbers disjoint from $\mathcal P$.
Note that if a conjugating sequence $((m_0,w_0)=(m,v),(m_1,w_1),\dots, (m_s,w_s)=(n,v'))$ exists then the integers $m_i$ have their prime divisors in $\mathcal P$. Call $\Z_{\mathcal P}$ the set of integers with prime divisors in $\mathcal P$.
Call $\pi$ the map from $\Z_{\mathcal P}\times \mathcal V$ to $(\N^{k}\times \Z/2\Z)\times \N^r$ which maps $(z,v_j)$ to the couple $(s,t)$ where $s\in\N^{k}\times \Z/2\Z$ is the decomposition in prime factors of $z$ (with its sign), and $t\in \N^r$ is the vector $(\delta_{i,j})_i$ where $\delta$ is the Kronecker symbol.

Finding a conjugating sequence $((m_0,w_0)=(m,v),\dots, (m_s,w_s)=(n,v')$ is now equivalent to finding a sequence $(r_0=\pi(n),\dots, r_s=\pi(m))$ of elements of $\N^k\times \Z/2\Z\times \N^r$ such that for all $i<s$ there exists an edge ${e_{j_i}}$ verifying $r_i-\pi(\sigma_{j_i},i(e_{j_i})) \in \N^k\times \Z/2\Z\times N^r$ and $r_{i+1}= r_i +\pi(\tau_{j_i},t(e_{j_i})- \pi({\sigma_i},i(e_{j_i})$.

The problem to decide if such a sequence exists is proven to be solvable in \cite{Reut}.
\end{proof}

\begin{corollary}
 The multiple conjugacy problem in \GBS groups is solvable.
\end{corollary}
 
 \begin{proof}
 Let $G$ be a \GBS group.
 Take $A=(\AA_0,\dots,\AA_k)$ and $B=(\BB_0,\dots,\BB_k)$ to $k+1$-tuples of $G$. If $\langle\AA_0,\dots,\AA_k \rangle$ is not an elliptic subgroup, then the problem is already solved by theorem \ref{multiconj}. Otherwise, all $a_i$ belongs to a same vertex group $\langle a \rangle$, and obviously, we can to the same for $B$. Call $b$ a generator of the vertex group containing  $\langle\BB_0,\dots,\BB_k \rangle$. Call $a'$ the gcd of the $a_i$'s in  $\langle a \rangle$ and $b'$ the gcd of the $b_i$'s in  $\langle b \rangle$. For every $i$ there exists $m_i$ and $n_i$ such that $a_i=a'^{m_i}$ and $b_i=b'^{n_i}$. Then $A$ and $B$ are conjugate if and only if $a'$ and $b'$ are conjugate and $m_i=n_i$ for every $i$.
Finally theorem \ref{conjgbs} allows us to conclude.
 \end{proof}
 
\bibliographystyle{plain}
\bibliography{multipleconjugacy}

\begin{thebibliography}{1}

\bibitem{BaCaRo}
Gilbert Baumslag, Frank~B. Cannonito, Derek~J. Robinson, and Dan Segal.
\newblock The algorithmic theory of polycyclic-by-finite groups.
\newblock {\em J. Algebra}, 142(1):118--149, 1991.

\bibitem{BoMaVen}
O.~B{\sc{ogopolski}}, A.~M{\sc{artino}}, and E.~V{\sc{entura}}.
\newblock Orbit decidability and the conjugacy problem for some extensions of
  groups.
\newblock {\em Trans. Amer. Math. Soc.}, 362(4):2003--2036, 2010.

\bibitem{GrunSe}
Fritz Grunewald and Daniel Segal.
\newblock Conjugacy in polycyclic groups.
\newblock {\em Comm. Algebra}, 6(8):775--798, 1978.

\bibitem{Reut}
C.~R{\sc{eutenauer}}.
\newblock {\em Aspects math\'ematiques des r\'eseaux de P\'etri}.
\newblock \'Etudes et recherches en informatique. MASSON, Paris, 1989.

\bibitem{Sela97}
Z.~Sela.
\newblock Acylindrical accessibility for groups.
\newblock {\em Invent. Math.}, 129(3):527--565, 1997.

\bibitem{Serre}
J.-P. S{\sc{erre}}.
\newblock {\em Trees}.
\newblock Springer-Verlag, Berlin, 1980.

\end{thebibliography}

\vspace{0.3cm}
\begin{flushleft}
\noindent \textsc{Benjamin Beeker}\\
LMNO,\\
Universit\'e de Caen BP 5186\\
F 14032 Caen Cedex \\
France\\
{\tt benjamin.beeker@math.unicaen.fr}
\end{flushleft}
\end{document}